\documentclass[12pt,reqno]{article}

\usepackage[usenames]{color}
\usepackage[colorlinks=true,
linkcolor=webgreen, filecolor=webbrown,
citecolor=webgreen]{hyperref}

\definecolor{webgreen}{rgb}{0,.5,0}
\definecolor{webbrown}{rgb}{.6,0,0}

\usepackage{amssymb}
\usepackage{graphicx}
\usepackage{amscd}
\usepackage{lscape}
\usepackage{tikz}
\usepackage{tikz-cd}
\usepackage{pgfplots}

\usetikzlibrary{matrix}
\usetikzlibrary{fit,shapes}
\usetikzlibrary{positioning, calc}
\tikzset{circle node/.style = {circle,inner sep=1pt,draw, fill=white},
        X node/.style = {fill=white, inner sep=1pt},
        dot node/.style = {circle, draw, inner sep=5pt}
        }
\usepackage{tkz-fct}

\usepackage{amsthm}
\newtheorem{theorem}{Theorem}

\newtheorem{proposition}[theorem]{Proposition}

\newtheorem{conjecture}[theorem]{Conjecture}

\theoremstyle{definition}

\newtheorem{example}[theorem]{Example}

\usepackage{float}

\usepackage{graphics,amsmath}
\usepackage{amsfonts}
\usepackage{latexsym}
\usepackage{epsf}

\newcommand{\seqnum}[1]{\href{http://oeis.org/#1}{\underline{#1}}}

\begin{document}

\begin{center}
\vskip 1cm{\LARGE\bf Conjectures and results on some generalized Rueppel sequences} \vskip 1cm \large
Paul Barry\\
School of Science\\
Waterford Institute of Technology\\
Ireland\\
\href{mailto:pbarry@wit.ie}{\tt pbarry@wit.ie}
\end{center}
\vskip .2 in

\begin{abstract} In this note we use the analogy between the Catalan sequence and the Rueppel sequence to derive a variety of conjectures surrounding the Hankel transforms of a number of sequences closely related to the Rueppel sequence. Use is made of the representation of suitable generating functions by Stieltjes continued fractions. We define polynomial sequences by introducing parameters that define generalized Rueppel sequences, and we show that such polynomials have coefficient arrays that are Riordan arrays. Finally we conjecture the form of a product of Hankel transforms arising from the Rueppel sequence. \end{abstract}

\section{Introduction}
In a previous paper \cite{SomePB}, we studied aspects of the Rueppel sequence and we proposed a number of conjectures. Many of these conjectures have subsequently been proven \cite{SomeAHS}. In this note, we continue to study the Rueppel sequence, and we put forward a number of further conjectures and results. Roland Bacher, in a paper on the links between the paper-folding sequence and the Catalan numbers \cite{Bacher}, enunciates the principle ``Paperfolding=Catalan modulo $2\,\,+$ signs given by $2$-automatic sequences''. This paper is also inspired by this principle. We do not touch upon automaticity \cite{Automatic} here, but it should be understood as an underlying theme.
 
The Rueppel sequence $r_n$ \seqnum{A036987}, which begins
$$1, 1, 0, 1, 0, 0, 0, 1, 0, 0, 0,\ldots,$$ can be defined by 
$$r_n = C_n \,\bmod\,2.$$ Here, $C_n$ represents the $n$-th Catalan number 
$$C_n=\frac{1}{n+1} \binom{2n}{n}.$$ This is \seqnum{A000108} in the On-Line Encyclopedia of Integer Sequences \cite{SL1, SL2}. The sequence of Catalan numbers $C_n$ begins 
$$1,1,2,5,14,42,\ldots.$$ It is one of the most important sequences in combinatorics \cite{Stanley}. 
The Rueppel sequence $r_n$ has generating function 
$$r(x)=\sum_{n=0}^{\infty} x^{2^n-1},$$ while the Catalan numbers $C_n$ have their generating function given by 
$$c(x)=\frac{1-\sqrt{1-4x}}{2x}.$$ 

The Hankel transform of the Catalan numbers is the sequence $C_n^{(h)}=|C_{i+j}|_{0 \le i,j \le n}$, which begins 
$$1,1,1,1,\ldots.$$
In fact, $C_n^{(h)}=1$ for all $n$.

The Hankel transform of the Rueppel numbers $r_n^{(h)}=|r_{i+j}|_{0\le i,j \le n}$ begins
$$1, -1, -1, 1, 1, -1, -1, 1, 1, -1, -1,\ldots.$$
We have $r_n^{(h)}=(-1)^{\binom{n+1}{2}}$. We can consider $r_n^{(h)}$ to be a particular signed version of $C_n^{(h)}$.

This correspondence may be extended to other sequences related to the Catalan and the Rueppel sequences.

\begin{example} The sequence with generating function $\frac{1}{c(x)}$ expands to give the sequence that begins
$$1, -1, -1, -2, -5, -14, -42, -132, -429, -1430, \ldots.$$ The generating function of this sequence can also be expressed as $1-xc(x)$. The Hankel transform of this sequence begins
$$1, -2, 3, -4, 5, -6, 7, -8, 9, -10, \ldots,$$ which is $(-1)^n (n+1)$ with generating function
$\frac{1}{(1+x)^2}$.

The sequence with generating function $1-xr(x)$, which begins
$$1, -1, -1, 0, -1, 0, 0, 0, -1, 0, 0,\ldots,$$ has a Hankel transform that begins
$$1, -2, 3, 2, -3, 4, 3, 2, -3, 4, -5,\ldots.$$
We regard this as a ``Rueppel'' version of the sequence $1, -2, 3, -4, 5, -6, 7,\ldots$. Note that both these Hankel transforms give the sequence
$$1,0,1,0,1,0,1,0,\ldots$$ when taken modulo $2$.
\end{example}

\section{Continued fractions}
Additional insight can be achieved by looking at the continued fraction expressions \cite{Wall} for the generating functions of the sequences considered so far.
Thus the generating function of the Catalan numbers $c(x)$ can be expressed as
$$c(x)=\cfrac{1}{1-\cfrac{x}{1-\cfrac{x}{1-\cfrac{x}{1-\cdots}}}}.$$
We express this as $$c(x)=\mathcal{S}(1,1,1,\ldots),$$ where the notation $\mathcal{S}(a,b,c,\ldots)$ indicates that this is a Stieltjes continued fraction with the indicated parameters.
Alternatively, we have
$$c(x)=\cfrac{1}{1-x-\cfrac{x^2}{1-2x-\cfrac{x^2}{1-2x-\cfrac{x^2}{1-2x-\cdots}}}}.$$ We can express this as
$$c(x)=\mathcal{J}(1,2,2,2,\ldots; 1,1,1,\ldots),$$ where the notation $\mathcal{J}(a,b,c,\ldots;u,v,w,\ldots)$ indicates that this is a Jacobi continued fraction.

The continued fraction expression for the generating function of the Rueppel sequence is given by
$$r(x)=\mathcal{S}(1, -1, -1, 1, -1, 1, -1, 1, 1, -1, 1, \ldots).$$ The coefficient sequence in this expression
is given by $2 a_n-1$, where $a_n=\seqnum{A088567}(n+2) \bmod 2$. Here, the sequence \seqnum{A088567}, which begins
$$1, 1, 1, 2, 2, 3, 4, 5, 6, 7, 9, 10, 13, 14, 18,\ldots,$$ counts the number of ``non-squashing'' partitions \cite{Different, Squash} of $n$ into distinct parts. It also counts the number of binary partitions of $n$. The generating function of \seqnum{A088567}$(n+2)$ is given by
$$\frac{1}{1-x}+\sum_{k=1}^{\infty} \frac{x^{3\cdot 2^{k-1}-2}}{\prod_{j=0}^k 1-x^{2^j}}.$$ 

Alternatively, we have
$$r(x)=\mathcal{J}(1, -2, 0, 0, 2, 0, -2, 0, 2, -2, 0,\ldots;-1,-1,-1,\ldots)$$ as a Jacobi continued fraction. 
The sequence beginning $1, -2, 0, 0, 2, 0, -2, 0, 2, -2, 0,\ldots$ is $-\seqnum{A110036}(n+1)$. We have 
$|\seqnum{A110036}(n+2)|= 2(\seqnum{A088567}(n) \bmod 2)$.
Knowing the parameters of a Stieltjes or a Jacobi continued fraction allows us to find an expression for the Hankel transform of the expansion of the generating function given by the  continued fraction (see Appendix for more details on continued fractions and Hankel transforms).
\begin{example} The sequence 
$$1,-1,-1,-2,-5,-14,-42,\ldots$$ has its generating function given by 
$$1-xc(x)=\mathcal{S}\left(-1, 2, \frac{1}{2},\frac{3}{2},\frac{2}{3},\frac{4}{3}, \frac{3}{4} , \frac{5}{4},\ldots\right).$$ 
Alternatively, we have 
$$1-xc(x)=\mathcal{J}\left(-1, \frac{5}{2}, \frac{13}{6}, \frac{25}{12},\frac{41}{20}, \frac{61}{30},\ldots; - 2, \frac{3}{4},\frac{8}{9}, \frac{15}{16}, \frac{24}{25}, \frac{35}{36},\ldots\right).$$
On the other hand, we have 
$$1-xr(x)=\mathcal{S}\left(-1, 2, - \frac{1}{2}, - \frac{3}{2}, \frac{2}{3}, - \frac{2}{3}, \frac{3}{2}, - \frac{3}{2}, \frac{2}{3}, \frac{4}{3},\ldots\right),$$ and 
$$1-xr(x)=\mathcal{J}\left(-1,\frac{3}{2},-\frac{5}{6},\frac{5}{6},-\frac{5}{6}, \frac{7}{12},\ldots;- 2,\frac{3}{4},-\frac{4}{9},-\frac{9}{4},\frac{8}{9}, -\frac{9}{16},\ldots\right).$$  
\end{example}
\section{Some new conjectures}
In the Introduction, we saw that the Hankel transforms of the Rueppel sequence and the sequence with generating function $1-xr(x)$ are ``Rueppel versions'' of the Hankel transforms of the Catalan numbers and the sequence with generating function $1-xc(x)$. We now consider two other ``Rueppel analogs'' of Catalan related sequences.
\begin{example} $\mathbf{1-x+x^2 r(x^2)}$. We first look at the Catalan analog. Thus we consider the sequence 
$$1,-1,1,0,1,0,1,0,2,0,5,0,14,0,\ldots$$ with generating function 
$$g(x)=1-x+x^2 c(x^2).$$ 
The Hankel transform of this sequence begins 
$$1, 0, -1, -2, -3, -4, -5, -6, -7, -8, -9,\ldots,$$ with generating function 
$$1-\frac{x^2}{(1-x)^2}=\frac{1-2x}{(1-x)^2}.$$ 
We now consider the Hankel transform of the sequence
$$1, -1, 1, 0, 1, 0, 0, 0, 1, 0, 0, 0, 0, 0, 0, 0, 1, 0,\ldots$$ with generating function 
$$1-x+x^2 r(x^2).$$
We find that this Hankel transform begins 
$$1, 0, -1, 0, 1, 2, -1, 0, 1, 2, 3, -2, 1, 2, -1, 0, 1, 2, 3, -2, -3,\ldots.$$ 
\begin{conjecture} The Hankel transform of the sequence with generating function $1-x+x^2r(x^2)$ is a signed version of \seqnum{A037834}.
\end{conjecture}
The sequence \seqnum{A037834} is given by the numbers of $i$ such that $|d(i) - d(i-1)| = 1$, where $\sum_{i=0}^m d(i)2^i$ is the base-$2$ representation of $n$.
\end{example}
\begin{example} $\mathbf{1+x-x^2 r(x^2)}$. In this example, we look at the sequence 
$$1, 1, -1, 0, -1, 0, 0, 0, -1, 0, 0, 0, 0, 0, 0, 0, -1, 0,\ldots$$ with generating function 
$$1+x-x^2 r(x^2).$$ 
Since we have $r(x)=1+x r(x^2)$, we deduce that 
$$1+x-x^2r(x)=1+2x-rg(x)=1-xgr(x)+2x.$$ Calculating, we find that the Hankel transform of this sequence begins 
$$1, -2, 3, 2, -3, 4, 3, 2, -3, 4, -5, -4, -3, 4, 3, 2, \ldots.$$  
\begin{proposition} The sequence with generating function $1+x-x^2r(x^2)=1-xr(x)+2x$ has the same Hankel transform as the sequence with generating function $1-xr(x)$.
\end{proposition}
\begin{proof} The expansion of $1-xr(x)+2x$ is $(-1)^n$ times the expansion of $1-xr(x)$. Sequences with this property share the same Hankel transforms.
\end{proof}

In the Catalan ``domain'', we find that the sequence 
$$1,1,-1,0,-1,0-2,0,-5,0,-14,0,\ldots$$ with generating function $1+x-x^2c(x^2)$ has the Hankel transform 
$$1,-2,3,-4,5,-6,\ldots.$$ 

\end{example}
\begin{example} $\mathbf{\frac{1}{1+xr(x)}}$. We first consider the sequence with generating function $\frac{1}{1+x c(x)}$. This is sequence \seqnum{A126983}
$$1, -1, 0, -1, -2, -6, -18, -57, -186, -622, -2120,\ldots$$ with generating sequence $\frac{1}{1+xc(x)}=1-F(x)$, where $F(x)$ is the generating function of the Fine numbers \seqnum{A000957}. The Hankel transform of \seqnum{A126983} begins
$$1,1,1,\ldots.$$ 
Looking now at the sequence \seqnum{A339422}
$$1, -1, 0, 1, -2, 2, 0, -3, 4, -2, -2, 6, -6, 0, 8, -11,\ldots$$ with generating function $\frac{1}{1+xr(x)}$, a calculation shows that its Hankel transform $h_n$ begins 
$$1, -1, 1, 1, -1, 1, 1, 1, -1, 1, -1, -1, -1, 1, 1, 1, -1, 1, -1, -1, 1,\ldots.$$ 
\begin{conjecture}
We have 
$$ \left|\frac{(-1)^{\binom{n}{2}}-h_n}{2}\right|=\seqnum{A268411},$$ 
where \seqnum{A268411} gives the parity of the number of runs of $1$'s in the binary representation of $n$.
\end{conjecture}
Thus we conjecture that  $\frac{(-1)^{\binom{n}{2}}-h_n}{2}$ is a signed version of the sequence \seqnum{A268411}. This signed sequence begins 
 $$0, 1, -1, -1, 1, 0, -1, -1, 1, 0, 0, 0, 1, 0, -1, -1,\ldots.$$ 
There is evidence to suggest that the positions of the sign changes are governed by the sequence \seqnum{A043725}, which gives the numbers $n$ such that the number of runs in base $2$ representation of $n$ is congruent to $1 \bmod 4$.

\end{example}
\begin{example} $\mathbf{1-\frac{x}{r(x^2)}}$. We consider first the sequence 
$$1, -1, 0, 1, 0, 1, 0, 2, 0, 5, 0,\ldots$$ with generating function 
$$1-\frac{x}{c(x^2)}=\frac{2-x-x\sqrt{1-4x^2}}{2}.$$ The Hankel transform of this sequence begins 
$$1, -1, -1, 4, 1, -9, -1, 16, 1, -25, -1, 36, 1, -49, -1, 64, 1, -81, -1, 100,\ldots.$$ 
Turning now to the sequence with generating function $1-\frac{x}{r(x^2)}$, this sequence begins 
$$1, -1, 0, 1, 0, -1, 0, 2, 0, -3, 0, 4, 0, -6, 0, 10, 0,\ldots.$$ The Hankel transform of this sequence begins 
$$1, -1, -1, 0, 1, -1, -1, 0, 1, -1, -1, 0, 1, -1, -1, 0, 1, -1, -1, 0, \ldots.$$
We then have the following conjecture.
\begin{conjecture} The Hankel transform of the sequence with generating function $1-\frac{x}{r(x^2)}$ is the periodic sequence 
$$1, -1, -1, 0, 1, -1, -1, 0, 1, -1, -1, 0,\ldots.$$ 
\end{conjecture}
We note further that the sequence with generating function $1-x\left(1-\frac{x}{c(x^2)}\right)$ has a Hankel transform that begins 
$$1, -2, -1, -1, 7, 11, 38, 51, 115, 144, 269,\ldots$$ which taken modulo $2$ begins 
$$1, 0, 1, 1, 1, 1, 0, 1, 1, 0, 1, 1, 1, 1, 0, 1, 1, 0,\ldots.$$ 
We conjecture that taken modulo $2$, this Hankel transform yields the periodic sequence with generating function $$\frac{1+x^2+x^3+x^4+x^5+x^7}{1-x^8}.$$ 
Turning now to the sequence with generating function $1-x\left(1-\frac{x}{r(x^2)}\right)$, we find that it has a Hankel transform that begins 
$$1, -2, -1, 1, 1, 1, -2, 1, 1, 2, 1, -1, 1, 1, -2, 1, 1, 2, 1, -1, -1, -1, \ldots.$$ 
\begin{conjecture} The Hankel transform of the sequence with generating function $1-x+\frac{x^2}{r(x^2)}$, taken modulo $2$, yields the periodic sequence $\overline{1,0,1,1,1,1,0,1}$. 
\end{conjecture}
\end{example}
\begin{example} $\mathbf{x+\frac{1}{r(x^2)}}$. We consider the sequence 
$$1, 1, -1, 0, -1, 0, -2, 0, -5, 0, -14, 0, -42,0,\ldots$$ with generating function $x+\frac{1}{c(x^2)}$. This sequence then has the Hankel transform 
$$1,-2,3,-4,5,-6,\ldots$$ with generating function $\frac{1}{(1+x)^2}$. The sequence with generating function $x+\frac{1}{r(x^2)}$ begins 
$$1, 1, -1, 0, 1, 0, -2, 0, 3, 0, -4, 0, 6,0,\ldots.$$
Calculation shows that the Hankel transform of this sequence begins 
$$1, -2, -1, 2, -3, -2, -1, 2, -3, 4, 3, 2, -3,\ldots.$$
The sequence \seqnum{A005811}, begins 
$$0, 1, 2, 1, 2, 3, 2, 1, 2, 3, 4, 3, 2, \ldots.$$ This sequence counts the number of runs in the binary expansion of $n, (n>0)$. Alternatively, it gives the number of $1$'s in the Gray code for $n$.
We can formulate the following conjecture.
\begin{conjecture} The Hankel transform of the sequence with generating function $x+\frac{1}{r(x^2)}$ is a signed version of \seqnum{A005811}$(n+1)$.
\end{conjecture}
\end{example}
\begin{example} $\mathbf{1-x+\frac{x^2}{1-x^2 r(x^2)}}$. The sequence 
$$1, -1, 1, 0, 1, 0, 2, 0, 5, 0, 14, 0, 42, 0, \ldots$$ with generating function 
$$1-x+\frac{x^2}{1+x^2 c(x^2)}=1-x+x^2c(x^2)$$ has its Hankel transform given by 
$$1,0,-1,-2,-3,-4,-5,-6,\ldots,$$ with generating function 
$$1-\frac{x^2}{(1-x)^2}=\frac{1-2x}{(1-x)^2}.$$ 

The sequence which begins 
$$1, -1, 1, 0, 1, 0, 2, 0, 3, 0, 6, 0, 10, 0, 18, 0, 31, 0, \ldots$$ with generating function 
$$1-x+\frac{x^2}{1+x^2 r(x^2)}$$ has a Hankel transform which begins 
$$1, 0, -1, -2, 1, 2, 3, -2, 1, 2, 3,\ldots.$$
We then have the following conjecture. 
\begin{conjecture} The Hankel transform of the sequence with generating function $1-x+\frac{x^2}{1+x^2 r(x^2)}$ is given by \seqnum{A005811}$(n-1)$. 
\end{conjecture}
Further insight into this example may be gained as follows. The Hankel transform of the parameterized sequence 
$$1, s, 1, 0, 1, 0, 2, 0, 5, 0, 14, 0, 42, 0, \ldots$$ is given by 
$$[1, 1, 0, 0, -1, -1, -2, -2, -3, -3, -4,\ldots]+s^2[0, -1, -1, -2, -2, -3, -3, -4, -4, -5, -5,\ldots].$$ 
In similar fashion, the Hankel transform of the sequence 
$$1, s, 1, 0, 1, 0, 2, 0, 3, 0, 6, 0, 10, 0, 18, 0, 31, 0, \ldots$$ is given by 
$$[1, 1, 0, 0, -1, -1, 0, 0, -1, -1, 0,\ldots]+s^2[0, -1, -1, -2, 2, 3, 3, -2, 2, 3, 3,\ldots].$$ 
We note that in both cases the sign of $s$ does not matter as the Hankel transform depends on $s^2$. 
\begin{conjecture} The Hankel transform of the sequence with generating function $1\pm x+\frac{x^2}{1+x^2 r(x^2)}$ is given by \seqnum{A005811}$(n-1)$.
\end{conjecture}
When $s=0$, we have that the Hankel transform of the sequence with generating function $1+\frac{x^2}{1+x^2 r(x^2)}$ begins
$$1, 1, 0, 0, -1, -1, 0, 0, -1, -1, 0, 0, 1, -1, 0, 0, -1, -1, 0, 0, 1, 1, 0, 0, \ldots.$$ 
It is interesting to note that the sequence beginning
$$1, 1, 0, 1, 0, 2, 0, 3, 0, 6, 0, 10, 0, 18, 0, 31, 0, \ldots$$ with generating function $1+\frac{x}{1+x^2r(x^2)}$ has its generating function given by 
$$\mathcal{J}(1, -2, 2, 0, 0, -2, 0, 2, 0, -2, 2,\ldots; -1,-1,-1,\ldots),$$ while that of the Rueppel sequence with generating function $r(x)$ is given by 
$$\mathcal{J}(1, -2, 0, 0, 2, 0, -2, 0, 2, -2, 0,\ldots;-1,-1,-1,\ldots).$$ The Hankel transform of both these sequences is equal to $(-1)^{\binom{n+1}{2}}$. 

The sequence that begins 
$$1, 1, 0, 2, 0, 5, 0, 14, 0, 42, 0,\ldots$$ with generating function
$$\frac{1-(x^2-x)c(x^2)(x-x^2)}{1-x^2 c(x^2)}=\frac{1-2x^2+2x^3-\sqrt{1-4x^2}}{2x^3}$$ has a Hankel transform that begins
$$1, -1, -4, 1, 9, -1, -16, 1, 25, -1, -36, 1, 49,\ldots.$$ 
The sequence that begins $$1, 1, 0, 2, 0, 3, 0, 6, 0, 10, 0, 18, 0, 31, 0, 56, 0, 98, 0, 174, 0,\ldots,$$ whose generating function is given by $\frac{1-(x^2-x)r(x^2)(x-x^2)}{1-x^2 r(x^2)}$ has a Hankel transform $n_n$ that begins
$$1, -1, -4, 1, 9, -1, -4, 1, 9, -1, -16, 1, 9, -1, -4, 1, 9, -1, -16, 1, 25,\ldots.$$ 
The sequence $\sqrt{|h_{2n}|}$ begins 
$$1, 2, 3, 2, 3, 4, 3, 2, 3, 4, 5, 4, 3, 4, 3, 2, 3,\ldots.$$ We then have the following conjecture. 
\begin{conjecture} The Hankel transform $h_n$ of the sequence with generating function $$\frac{1-(x^2-x)r(x^2)(x-x^2)}{1-x^2 r(x^2)}$$ satisfies $\sqrt{|h_{2n}|}$=\seqnum{A088748}.
\end{conjecture}
\end{example}
\section{Generalized Rueppel sequences and Riordan arrays}
In this section, we consider two generalizations of the Rueppel sequence. 
These are the sequences that begin 
$$1, b, 0, b, 0, 0, 0, b, 0, 0, 0,\ldots$$ with generating function 
$$r_b(x)=b r(x)-(b-1)$$ 
and 
$$1, c, 0, b, 0, 0, 0, b, 0, 0, 0,\ldots$$ with generating function with generating function 
$$r_{b,c}(x)=b r(x)-(b-1)+(c-b)x.$$ Thus 
$$r_{b,c}=1+cx+b(x^3+x^7+x^{15}+x^{31}+\cdots).$$ 
Note that $r_b(x)=r_{b,b}(x)$ and $r(x)=r_{1,1}(x)$. 
We find that 
$$r_{b,c}(x)=\mathcal{S}\left(c,-c,-\frac{b}{c^2},\frac{b}{c^2},-c,c,-\frac{1}{b},\frac{1}{b},c,-c,\frac{b}{c^2},-\frac{b}{c^2},-c,c,-\frac{1}{b},\frac{1}{b},c,-c,\ldots\right).$$
In particular,
$$r_{b,1}=\mathcal{S}\left(1,-1,-b,b,-1,1,-\frac{1}{b},\frac{1}{b},1,-1,b,-b,-1,1,-\frac{1}{b},\frac{1}{b},1,-1,-b,b,-1,\ldots\right)$$ and 
$$r_b(x)=\mathcal{S}\left(b,-b,-\frac{1}{b},\frac{1}{b},-b,b,-\frac{1}{b},\frac{1}{b},b,-b,\frac{1}{b},-\frac{1}{b},-b,b,-\frac{1}{b},\frac{1}{b},b,-b,-\frac{1}{b},\frac{1}{b},\ldots\right).$$ 
\begin{conjecture} Let $s_{b,c}(n)$ be the sequence of Stieltjes parameters in 
$$r_{b,c}(x)=\mathcal{S}\left(c,-c,-\frac{b}{c^2},\frac{b}{c^2},-c,c,-\frac{1}{b},\frac{1}{b},c,-c,\frac{b}{c^2},-\frac{b}{c^2},-c,c,-\frac{1}{b},\frac{1}{b},c,-c,\ldots\right).$$
Then we have 
\begin{itemize}
\item $s(8n)=s(8n+5)=c$\\
\item $s(8n+1)=s(8n+4)=-c$\\
\item $s(8n+3)=-s(8n+2)=(-1)^n \frac{b}{c^2}$\\
\item $s(8n+7)=-s(8n+6)=(2 P(n)-1)\frac{1}{b}$,
\end{itemize} 
where $P(n)$ is the paper-folding sequence \seqnum{A014577}.
\end{conjecture}
This allows us to conjecture that $s(n)$ has the following closed form:
$$s_{b,c}(n)=(1+(-1)^n)\left(\frac{b\left(1-P\left(\frac{n-2}{4}\right)\right)(1-(-1)^{\frac{n}{2}})}{4c^2}+\frac{c(-1)^{\frac{n}{4}}(1+(-1)^{\frac{n}{2}})}{4}\right)$$
$$\quad\quad -(1-(-1)^n)\left(\frac{b\left(1-P\left(\frac{n-3}{4}\right)\right)(1-(-1)^{\frac{n-1}{2}})}{4c^2}+\frac{c(-1)^{\frac{n-1}{4}}(1+(-1)^{\frac{n-1}{2}})}{4}\right).$$
The Hankel transform of the generalized Rueppel sequence 
$$1, c, 0, b, 0, 0, 0, b, 0, 0, 0,\ldots$$ then begins 
$$1,-c^2,-b^2,b^4,b^4,-b^4 c^2,-b^6,b^8,b^8,-b^8 c^2,-b^{10},b^{12},b^{12}\ldots.$$ 
We have 
$$r_{b,c}(x)=\cfrac{1}{1-\cfrac{s_{b,c}(0)x}{1-\cfrac{s_{b,c}(1)x}{1-\cdots}}}=\frac{1}{1-s_{b,c}(0)x r_{b,c}^{(1)}(x)}.$$
We are interested in exploring the sequence with generating function $r_{b,c}^{(1)}(x)$. Thus we have 
$$r_{b,c}^{(1)}(x)=\cfrac{1}{1-\cfrac{s_{b,c}(1)x}{1-\cfrac{s_{b,c}(2)x}{1-\cdots}}}.$$ We also have
$$r_{b,c}(x)=\frac{1}{1-cxr_{b,c}^{(1)}(x)} \Longrightarrow r_{b,c}^{(1)}(x)=\frac{1+\frac{b}{c}(x^2+x^6+\ldots)}{1+cx+b(x^3+x^7+\ldots)}.$$ This expands to give a sequence which begins
$$1,-c,\frac{b+c^3}{c}, -2b-c^3,\ldots.$$ The two cases $c=b$ and $c=1$ are of most interest to us. 

\textbf{Case $c=b$}. In this case, the generating function $r_b(x)$ is given by 
 $$r_b(x)=\frac{1+x^2+x^6+x^{14}+\cdots}{1+bx+bx^3+bx^7+bx^{15}+\cdots}.$$ This expands to give the polynomial sequence that begins
$$1,-b,b^2+1,-b \left(b^2+2\right),b^2 \left(b^2+3\right),-b \left(b^4+4 b^2+1\right),b^6+5 b^4+3 b^2+1,\ldots.$$ This polynomial sequence in $b$ has a coefficient matrix that begins 
$$\left(
\begin{array}{ccccccccc}
 1 & 0 & 0 & 0 & 0 & 0 & 0 & 0 & 0 \\
 0 & -1 & 0 & 0 & 0 & 0 & 0 & 0 & 0 \\
 1 & 0 & 1 & 0 & 0 & 0 & 0 & 0 & 0 \\
 0 & -2 & 0 & -1 & 0 & 0 & 0 & 0 & 0 \\
 0 & 0 & 3 & 0 & 1 & 0 & 0 & 0 & 0 \\
 0 & -1 & 0 & -4 & 0 & -1 & 0 & 0 & 0 \\
 1 & 0 & 3 & 0 & 5 & 0 & 1 & 0 & 0 \\
 0 & -2 & 0 & -6 & 0 & -6 & 0 & -1 & 0 \\
 0 & 0 & 4 & 0 & 10 & 0 & 7 & 0 & 1 \\
\end{array}
\right).$$
We have the following result which identifies this coefficient array as a Riordan array \cite{book, SGWW}.
\begin{proposition} The coefficient array for the expansion of $r_b(x)$ is given by the Riordan array 
$$\left(r(x^2), -xr(x^2)\right).$$ 
\end{proposition}
\begin{proof} The bivariate generating function of the Riordan array $\left(r(x^2), -xr(x^2)\right)$ is given by $$\frac{r(x^2)}{1+xy r(x^2)}.$$ 
Specializing $y$ to $b$ then gives 
$$\frac{r(x^2)}{1+bxr(x^2)}=\frac{1+x^2+x^6+x^{14}+\cdots}{1+bx(1+x^2+x^6+\ldots)}=\frac{1+x^2+x^6+\ldots}{1+bx+bx^3+bx^7+\ldots}.$$
\end{proof}
The row sums of $(r(x^2), -xr(x^2))$ will have generating function $\frac{r(x^2)}{1+x r(x^2)}$. This corresponds to $b=1$. Noting that $r(x^2)=\frac{r(x)-1}{x}$, we have that 
$$\frac{r(x^2)}{1+x r(x^2)}=\frac{1}{x}\frac{r(x)-1}{r(x)}=\frac{r(x^2)}{r(x)}.$$ 
The row sums begin 
$$1, -1, 2, -3, 4, -6, 10, -15, 22, -34, 52,\ldots.$$ By the form of the generating function, this is the INVERT$(-1)$ transform of the expansion of $r(x^2)$, which is $r_{n+1}$. These two sequences thus share the same Hankel transform. This Hankel transform begins 
$$1, 1, -1, -1, -1, 1, -1, -1, -1, -1, 1, -1, -1, 1,\ldots.$$
\textbf{Case $c=1$}. In this case, we have 
$$r_{b,1}(x)=\frac{1+b(x^2+x^6+x^{14}+\cdots)}{1+x+bx^3+bx^7+bx^{15}+\cdots}.$$
This expands to give the polynomial sequence in $b$ that begins 
$$1,-1,b+1,-2 b-1,3 b+1,-b^2-4 b-1,3 b^2+6 b+1,-6 b^2-8 b-1,b^3+10 b^2+10 b+1,\ldots.$$ 
The coefficient array of this polynomial sequence in $b$ then begins 
$$\left(
\begin{array}{ccccccccc}
 1 & 0 & 0 & 0 & 0 & 0 & 0 & 0 & 0 \\
 -1 & 0 & 0 & 0 & 0 & 0 & 0 & 0 & 0 \\
 1 & 1 & 0 & 0 & 0 & 0 & 0 & 0 & 0 \\
 -1 & -2 & 0 & 0 & 0 & 0 & 0 & 0 & 0 \\
 1 & 3 & 0 & 0 & 0 & 0 & 0 & 0 & 0 \\
 -1 & -4 & -1 & 0 & 0 & 0 & 0 & 0 & 0 \\
 1 & 6 & 3 & 0 & 0 & 0 & 0 & 0 & 0 \\
 -1 & -8 & -6 & 0 & 0 & 0 & 0 & 0 & 0 \\
 1 & 10 & 10 & 1 & 0 & 0 & 0 & 0 & 0 \\
\end{array}
\right).$$ 
We then have the following proposition.
\begin{proposition} The coefficient array of $r_{b,1}(x)$ is given by the stretched Riordan array 
$$\left(\frac{-1}{1+x}, \frac{-x^3 r(x^4)}{1+x}\right),$$ 
with its first row removed.
\end{proposition} 
\begin{proof} 
The bivariate generating function of the stretched Riordan array $\left(\frac{-1}{1+x}, \frac{-x^3 r(x^4)}{1+x}\right)$ is given by 
$$\frac{\frac{-1}{1+x}}{1+y \frac{x^3r(x^4)}{1+x}}=\frac{-1}{1+x+x^3yr(x^4)}.$$ 
Thus the generating function of the array obtained by removing the first row $(-1,0,0,0,\ldots)$ is given by 
$$\frac{1}{x}\left(\frac{-1}{1+x+x^3yr(x^4)}+1\right)=\frac{1+x^2yr(x^4)}{1+x+x^3yr(x^4)}.$$ 
Specializing $y$ to $b$ gives us 
$$\frac{1+bx^2r(x^4)}{1+x+bx^3r(x^4)}=\frac{1+bx^2(1+x^4+x^{12}+\ldots)}{1+x+bx^3(1+x^4+x^{12}+\ldots)}=\frac{1+bx^2+bx^6+\ldots}{1+x+bx^3+bx^7+\ldots}.$$
\end{proof}

\section{$1-r_n$, $r_{n+1}-r_n$, and the Josephus problem}
The $1$'s-complement of the Rueppel sequence $1-r_n$, \seqnum{A043545} which begins 
$$0, 0, 1, 0, 1, 1, 1, 0, 1, 1, 1,\ldots,$$ is of interest in itself for a number of reasons.
\begin{example} \textbf{The Josephus problem}. In this example, we consider the sequence $1-r_{n+2}$ which begins
$$1, 0, 1, 1, 1, 0, 1, 1, 1, 1, 1, 1, 1, 0, 1, 1, 1, 1, 1, 1, 1,\ldots.$$ The $0$'s occur at locations determined by the expansion of $\frac{1+2x}{(1-x)(1-2x)}$ which begins 
$$1, 5, 13, 29, 61, 125, 253, 509, 1021, 2045, 4093,\ldots.$$ This is the sequence with general term $4\cdot2^n-3$, \seqnum{A036563}. At these locations $i$ we now place a $-\frac{i+1}{2}$ instead of $0$, and we prepend $1,0$ to the resulting sequence. This gives us the sequence that begins
$$1, 0, 1, -1, 1, 1, 1, -3, 1, 1, 1, 1, 1, 1, 1, -7, 1, 1, 1, 1, 1, 1, 1,\ldots.$$ 
The partial sums of this sequence then begin 
$$1, 1, 2, 1, 2, 3, 4, 1, 2, 3, 4, 5, 6, 7, 8, 1, 2,\ldots.$$ This is \seqnum{A062050}, whose $n$-th term is given by $2+ n - 2^{\lfloor \log_2(n+1) \rfloor}$. 
Multiplying (termwise) the sequence 
$$1, 0, 1, -1, 1, 1, 1, -3, 1, 1, 1, 1, 1, 1, 1, -7, 1, 1, 1, 1, 1, 1, 1,\ldots, $$ by the sequence 
$$1,2,2,2,2,2,2,2,\ldots$$ gives us the sequence that begins 
$$1,0,2, -2, 2, 2, 2, -6, 2, 2, 2, 2, 2, 2, 2, -14, 2, 2, 2, 2, 2,\ldots.$$ 
The partial sums of this sequence begin
$$1, 1, 3, 1, 3, 5, 7, 1, 3, 5, 7, 9, 11, 13, 15, 1, 3,\ldots.$$ 
This is \seqnum{A006257}, the solution to the Josephus problem where every second element is chosen. The general term of this sequence is $3+ 2(n - 2^{\lfloor \log_2(n+1) \rfloor})$.
\end{example}
Out interest lies in the Hankel transforms of the complement of the Rueppel numbers $1-r_n$ and that of the first differences of the Rueppel numbers $r_{n+1}-r_n$. 
\begin{conjecture} Let $H_n$ be the Hankel transform of $1-r_n$, and let $h_n$ be the Hankel transform of $r_{n+1}-r_n$. Then we have 
$$ |h_n| = \sqrt{|H_{n+1}|-|H_n|}.$$
\end{conjecture}
Note that a similar result is true for both the Catalan numbers and the Motzkin numbers. 
\section{A conjecture concerning a product of Hankel transforms}
We finish this note with the following conjecture. We let $h_n$ denote the Hankel transform of the Rueppel sequence $r_n$, and we let $H_n$ denote the Hankel transform of the once shifted Rueppel sequence $r_{n+1}$. 
\begin{conjecture} The sequence $\frac{1+(-1)^n h_n H_n}{2}$ is equal to the sequence \seqnum{A268411}$(n+1)$, where the sequence \seqnum{A268411} gives the parity of the number of runs of $1$'s in the binary representation of $n$.
\end{conjecture}
\section{Appendix}
In this Appendix we briefly look at the link between generating functions expressible as continued fractions and Hankel transforms, and we give some relevant information on Riordan arrays.

\textbf{\underline{Continued fractions and Hankel transforms}}

In the case that a sequence $a_n$ has a generating function  $g(x)$ expressible in the form of a Jacobi continued fraction,
$$g(x)=\cfrac{a_0}{1-\alpha_0 x-
\cfrac{\beta_1 x^2}{1-\alpha_1 x-
\cfrac{\beta_2 x^2}{1-\alpha_2 x-
\cfrac{\beta_3 x^2}{1-\alpha_3 x-\cdots}}}}$$ then
we have \cite{Kratt}
\begin{equation}\label{Kratt} h_n = a_0^{n+1} \beta_1^n\beta_2^{n-1}\cdots \beta_{n-1}^2\beta_n=a_0^{n+1}\prod_{k=1}^n
\beta_k^{n+1-k}.\end{equation}
Note that this is independent from $\alpha_n$.
In the case of $a_n$ having a generating function given by a Stieltjes continued fraction 
$$\cfrac{a_0}{1-\cfrac{\alpha_1 x}{1-\cfrac{\alpha_2 x}{1-\cdots}}},$$ then we have 
$$|a_{i+j}|_{0\le i,j \le n-1}=a_0^n (\alpha_1\alpha_2)^n(\alpha_3\alpha4)^{n-2}\cdots (\alpha_{2n-5}\alpha_{2n-4})^2 \alpha_{2n-3}\alpha_{2n-2}.$$ 

\textbf{\underline{Riordan arrays}}

A Riordan array \cite{book, SGWW} can be visualized as a matrix $(m_{n,k})$ whose elements are given by 
$$m_{n,k}=[x^n] g(x)f(x)^k,$$ 
where 
$$g(x)=g_0+g_1 x+ g_2 x^2 + \cdots$$ is a power series with $g_0 \ne 0$, and 
$$f(x)=f_1 x + f_2 x^2 + \cdots$$ where $f_1 \ne 0$. The coefficients may be from any ring that is of interest to us. In combinatorics, this is often $\mathbb{Z}$. 
Here, $[x^n]$ denotes the linear functional that extracts the coefficient of $x^n$ in the power series that it acts on. The set of such matrices forms a group, called the Riordan group. 

The bivariate generating function of the matrix $(m_{n,k})_{0 \le n,k \le \infty}$, as a two-dimensional array, is given by
$$\frac{g(x)}{1-yf(x)}.$$ The matrix $(m_{n,k})$ is a matrix representation of the couple $(g(x), f(x))$.
The ``fundamental theorem of Riordan arrays'' prescribes the action of $(g(x), f(x))$ on a power series. We have
$$(g(x), f(x))\cdot h(x)=g(x) h(f(x)).$$
This corresponds to multiplying the vector whose elements are the coefficients of the power series $h(x)$ by the matrix $(m_{n,k})$.

\bigskip
\hrule

\noindent 2010 {\it Mathematics Subject Classification}: 
Primary 11B50; Secondary 05A15, 11B83, 11C20, 11Y55, 15B36.
\noindent \emph{Keywords:} Rueppel sequence, paper-folding sequence, Hankel transform, Riordan array.

\bigskip
\hrule
\bigskip
\noindent (Concerned with sequences
\seqnum{A000108},
\seqnum{A005811},
\seqnum{A006257},
\seqnum{A014577},
\seqnum{A036563},
\seqnum{A036987},
\seqnum{A037834},
\seqnum{A043725},
\seqnum{A062050},
\seqnum{A088567},
\seqnum{A088748},
\seqnum{A110036},
\seqnum{A126983},
\seqnum{A268411},
\seqnum{A268411} and 
\seqnum{A339422}.)


\begin{thebibliography}{9}

\bibitem{SomeAHS} J.-P. Allouche, G.-N. Han and J. Shallit, On some conjectures of P. Barry, \emph{J. Number Theory}, \textbf{228} (2021), 108--132.
    
\bibitem{Automatic} J.-P. Allouche and J. Shallit, \emph{Automatic Sequences: Theory, Applications, Generalizations}, Cambridge University Press.
    
\bibitem{SomePB} P. Barry, Some observations on the Rueppel sequence and associated Hankel determinants, \url{https://arxiv.org/abs/2005.04066}.
    
\bibitem{book} P. Barry, \emph{Riordan Arrays: a Primer}, Logic Press, 2017.

\bibitem{Bacher} R. Bacher, Paperfolding and Catalan numbers, \url{https://arxiv.org/abs/math/0406340}. 

\bibitem{Kratt} C. Krattenthaler, Advanced determinant
    calculus: A complement, {\it Linear Algebra
    Appl.} \textbf{411} (2005), 68–-166.
    
\bibitem{Different} M. D. Hirschhorn and J. A. Sellers, A different view of $m$-ary partitions, \emph{Australas. J. Combin.}, \textbf{30} (2004), 193--196.

\bibitem{SGWW} L. W. Shapiro, S. Getu, W-J. Woan, and L.C. Woodson,
The Riordan group, \emph{Discr. Appl. Math.}, \textbf{34} (1991),
 229--239.

\bibitem{SL1} N. J. A.~Sloane, \emph{The
On-Line Encyclopedia of Integer Sequences}. Published electronically
at \texttt{http://oeis.org}, 2021.

\bibitem{Squash} N. J. A. Sloane and J. A. Sellers, On non-squashing partitions, \emph{Discrete Math.}, \textbf{294} (2005), 259--274.
    
\bibitem{SL2} N. J. A.~Sloane, The On-Line Encyclopedia of Integer
Sequences, \emph{Notices Amer. Math. Soc.}, \textbf{50} (2003),  912--915.

\bibitem{Stanley} R. P. Stanley, \emph{Catalan numbers}, Cambridge University Press, 2015.

\bibitem{Wall} H.~S. Wall, \emph{Analytic Theory of Continued Fractions}, AMS Chelsea Publishing, 2001.

\end{thebibliography}
\end{document}